\documentclass[12pt]{amsart}
   \usepackage{amsmath,amsthm}
   \usepackage{amsfonts}   
   \usepackage{amssymb}    
   \usepackage{url}
   \usepackage{pstricks}
      \usepackage{pstricks}
\usepackage{comment}
\advance \topmargin by -\headheight
\advance \topmargin by -\headsep
      
\evensidemargin \oddsidemargin
\usepackage[all]{xy}

\newtheorem{thm}{Theorem}[section]
\newtheorem{prop}[thm]{Proposition}
\newtheorem{lem}[thm]{Lemma}

\newtheorem{conj}[thm]{Conjecture}
\newtheorem{defn}[thm]{Definition}
\newtheorem{qu}[thm]{Question}

\theoremstyle{remark}
\newtheorem{rem}[thm]{Remark}
\newtheorem{ex}[thm]{Example}

\title{Contact perturbations of Reebless foliations are universally tight}
\author{Jonathan Bowden}
\address{Mathematisches Institut, Universit\"at Augsburg, Universit\"atsstr. 14, 86159 Augsburg, Germany}
\email{jonathan.bowden@math.uni-augsburg.de}
\date{\today}
\begin{document}
\begin{abstract}
We give a complete proof of the fact that a contact structure that is sufficiently close to a Reebless foliation is universally tight.
\end{abstract}

\maketitle
\section{Introduction}
In both the general theory of contact structures and that of foliations on $3$-manifolds one has a certain amount of flexibility due to the presence of overtwisted discs on the one hand and Reeb components on the other. This flexibility is borne out by the fact that any co-oriented plane field  is homotopic both to a contact structure and a foliation --  both contact structures and foliations satisfy an $h$-principle. On the other hand the construction of tight contact structures and foliations without Reeb components is a fundamental theme in both areas and the general existence questions has stimulated much research over several decades. 

The relationship between contact structures and foliations on $3$-manifolds was established by Eliashberg and Thurston \cite{ETh} who showed that any foliation apart from the product foliation on $S^2 \times S^1$ can be $C^0$-approximated by contact structures. They also observed that tautness of the foliation implies tightness of any sufficiently close contact structure. On the other hand the correct analogue of tightness for foliations ought to be the absense of Reeb components and in their book on confoliations Eliashberg and Thurston state that a contact structure $\xi$ that is sufficiently $C^0$-close to a Reebless foliation $\mathcal{F}$ is universally tight. However, their proof was incomplete and Colin \cite{Col} subsequently proved that a Reebless foliation can indeed be $C^0$-approximated by universally tight contact structures, but not that \emph{any} contact structure sufficiently closed to a Reebless foliation is necessarily universally tight. Recently Vogel \cite{Vog2} was able to give a proof under the additional assumption that all torus leaves have attractive holonomy. The aim of the present article is to give a complete proof of the original statement of the theorem as formulated in \cite{ETh}:
\begin{thm}\label{Reebless}
Let $\mathcal{F}$ be a Reebless foliation on a closed $3$-manifold $M$. Then there is a $C^0$-neighbourhood $\mathcal{U}_0$ of $T\mathcal{F}$ such that any contact structure $\xi$ in $\mathcal{U}_0$ is universally tight.
\end{thm}
The main steps in the proof are as follows: The fact that $\mathcal{F}$ is a Reebless foliation gives a decomposition of $M$ into two (possibly disconnected pieces) $M_{taut}, N_{Tor}$. This decomposition is such that $\mathcal{F}$ is taut on $M_{taut}$ and each component of $N_{Tor}$ is a thickened incompressible torus $T^2 \times [0,1]$ on which the foliation is transverse to the interval fibers. Any sufficiently small perturbation of $\mathcal{F}$ is then universally tight on each of these pieces. The key observation (Proposition \ref{Giroux}) is then that the contact structure on each torus piece is isotopic to one that is everywhere transverse to the $T^2$-slices. One then distinguishes two cases depending on whether the slopes of the characteristic foliations on the $T^2$-slices are constant or not. In the first case one can fill in the product piece with a foliation so that the contact structure looks like a perturbation of a taut foliation and is consequently universally tight. In the second case one can apply results of Colin \cite{Col0} for glueing universally tight contact structures along linearly foliated tori to conclude the proof.


Since the space of confoliations on a given closed $3$-manifold is the $C^0$-closure of the space of contact structures except in the case of the product foliation on $S^2 \times S^1$, Theorem \ref{Reebless} suggests a notion of tightness for confoliations, which might be called perturbation tightness or \emph{p-tightness}. Here a confoliation $\xi$ is $p$-tight if all positive contact structures are tight in some $C^0$-neighbourhood of $\xi$. Eliashberg and Thurston already proposed a definition of tightness for confoliations in \cite{ETh} which is in particular a generalisation of both tightness for contact structures and Reeblessness for foliations. However, Vogel \cite{Vog} has since shown that this notion of tightness for confoliations was too general, since the Thurston-Bennequin inequalities can be violated, and this led him to introduce the more restrictive notion of \emph{s-tightness}. Perhaps the correct definition of tightness for confoliations should capture the fact that if a given confoliation is not tight then it is a $C^0$-limit of overtwisted contact structures. In Section \ref{notions_of_tightness} we compare the various notions of tightness for confoliations and discuss the inclusions between them.


\subsection*{Acknowledgments:} We thank T. Vogel for helpful comments. The hospitality of the Max Planck Institute f\"ur Mathematik in Bonn, where part of this research was carried out, is also gratefully acknowledged. 

\subsection*{Conventions:} All manifolds, contact structures and foliations are smooth and oriented. Unless otherwise stated all manifolds will be closed and connected.

\section{Foliations, contact structures and confoliations}
A codimension-1 foliation $\mathcal{F}$ on a 3-manifold $M$ is a decomposition of $M$ into immersed surfaces called \emph{leaves} that is locally diffeomorphic to the decomposition of $\mathbb{R}^3$ given by the level sets of the projection to the $z$-axis. We will always assume that all foliations are smooth and cooriented. One can then define a global non-vanishing $1$-form $\alpha$ by requiring that 
$$\text{Ker}(\alpha) = T \mathcal{F} = \xi \subset T M.$$
By Frobenius' Theorem such a cooriented distribution is tangent to a foliation if and only if 
$$\alpha \wedge d \alpha \equiv 0$$
and in this case $\xi$ is called \emph{integrable}. By contrast a totally non-integrable plane field or \emph{contact structure} $\xi$ is a distribution such that $\alpha \wedge d \alpha$ is nowhere zero for any defining $1$-form with $\xi = \text{Ker}(\alpha)$. Unless specified otherwise all contact structures will be \emph{positive} with respect to the orientation on $M$ so that $\alpha \wedge d \alpha >0$. If $\alpha$ only satisfies the weaker inequality $\alpha \wedge d \alpha \geq 0$, then $\xi$ is called a (positive) \emph{confoliation}.

There is a fundamental dichotomy amongst contact structures between those that are tight and those that are overtwisted. Recall that a contact structure $\xi$ on manifold $M$ is called \emph{overtwisted} if there is an embedded disc $D \hookrightarrow M$ such that
$$TD|_{\partial D} = \xi |_{\partial D}.$$
If a contact structure $\xi$ admits no such disc then it is called \emph{tight}. A contact structure is \emph{universally tight} if its pullback to the universal cover $\widetilde{M} \to M$ is tight.

Recall, furthermore, that a foliation is \emph{taut} if each leaf admits a closed transversal, i.e. for each leaf $L$ of $\mathcal{F}$ there is a simple closed curve $\gamma$ intersecting $L$ that is everywhere transverse to $\mathcal{F}$. There are several equivalent formulations of tautness given by the following (cf.\ \cite{Sul}):
\begin{lem}\label{taut_def}
Let $\mathcal{F}$ be a foliation on $M$. Then the following are equivalent:
\begin{enumerate}
\item $\mathcal{F}$ is taut.
\item $M$ admits a dominating closed $2$-form $\omega$ with $\omega|_{\mathcal{F}} > 0$.
\item $M$ admits a metric so that all leaves of $\mathcal{F}$ are minimal surfaces.
\end{enumerate}
\end{lem}
The construction of the closed $2$-form in the implication $(1) \Longrightarrow (2)$ will be important. For this one takes a collection of transversals $\gamma_x$ through every point $x \in M$. Then a bump form on the $2$-disc $D$ gives a closed form on a small tubular neighbourhood $N(\gamma_x) \cong \gamma_x \times D$ that is non-negative on $\mathcal{F}$ and strictly positive on some open neighbourhood of $\gamma_x$. By compactness the sum of finitely many such forms will be positive on $\mathcal{F}$.

A slightly more general notion than tautness is that of a \emph{Reebless} foliation. Here a Reebless foliation is a foliation containing no Reeb components, where a Reeb component is a solid torus whose boundary is a leaf and whose interior is foliated by planes. Note that any taut foliation is Reebless since the boundary of a Reeb component is a null-homologous closed leaf, which thus admits no closed transversal. The existence of a Reebless foliations on a manifold has certain geometric consequences due to the following result of Novikov.
\begin{thm}[Novikov]
Let $\mathcal{F}$ be a Reebless foliation on a $3$-manifold. Then all leaves of $\mathcal{F}$ are incompressible and all transverse loops are essential in $\pi_1(M)$. Moreover, $\pi_2(M) = 0$ unless $\mathcal{F}$ is the product foliation on $S^2 \times S^1$.
\end{thm}

The relationship between contact structures and foliations is given by the following fundamental theorem of Eliashberg and Thurston.
\begin{thm}[Eliashberg-Thurston \cite{ETh}]\label{Eli_Th}
Let $\xi$ be an integrable plane field of class $C^2$ that is not tangent to the foliation by spheres on $S^2 \times S^1$. Then $\xi$ can be $C^0$-approximated by positive and negative contact structures.
\end{thm}
\noindent A version of Theorem \ref{Eli_Th} also holds for confoliations and shows that the space of confoliations is the $C^0$-closure of the space of (positive) contact structures unless $\xi$ is tangent to the product foliation on $S^2 \times S^1$.

One of the most important applications of Theorem \ref{Eli_Th} is that $C^0$-small contact perturbation of a taut foliation is (universally) tight, which in turn gives a large supply of tight contact structures in general. 
\begin{thm}[\cite{ETh}, Corollary 3.2.5]\label{taut _fill}
Let $\mathcal{F}$ be a taut foliation. Then any contact structure $\xi$ that is $C^0$-close to $T\mathcal{F}$ is universally tight. 
\end{thm}
\subsection{Surfaces in contact manifolds}
Given a surface $S$ in a contact manifold one can consider the induced singular foliation or \emph{characteristic foliation} on $S$ which we denote by $\xi(S)$. For a generic surface the singularities of this foliation will be non-degenerate and hence either elliptic or hyperbolic depending on whether the determinant of the linearisation of $\xi(S)$ at a singularity is positive or negative. After a further perturbation one can also assume that the linearisation at an elliptic singularity has real eigenvalues and we will always assume that this is the case. If $S$ is oriented, then each singularity carries a sign that is determined by whether the orientation of the contact structure agrees with that of the surface. 

An important notion for performing cut and paste operations in contact manifolds is that of a convex surface.
\begin{defn}[Convexity]
Let $\xi$ be a contact structure on a $3$-manifold $M$ and let $S \subset M$ be a closed, embedded surface. The surface is called convex if there is a vector field $X$ transverse to $S$ that preserves the contact structure on a small neighbourhood of $S$.
\end{defn}
\noindent Note that the contact structure in a neighbourhood of a convex surface is defined by a $1$-form
$$\lambda = \beta + fdt,$$
where $t$ is a normal coordinate given by the contact vector field $X$ and both $\beta$ and $f$ do not depend $t$. The set of points where $X$ is tangent to $\xi$ is an embedded submanifold transverse to $\xi(S)$ called the \textbf{dividing set} $\Gamma$ of $S$ and is well defined up to isotopy. The dividing set has the property that $\xi(S)$ is defined by a $1$-form $\beta$ such that $d\beta$ vanishes precisely on $\Gamma$ and the existence of such a defining form is then equivalent to $S$ being convex. Most importantly, convexity is a generic property so that after an initial $C^{\infty}$-perturbation we may always assume convexity (cf.\ \cite{Gir}). 

One has a very useful criterion for convexity in the case that the characteristic foliation satisfies the \emph{Poincar\'e-Bendixson property}: This means that the singularities of $\xi(S)$ are isolated and the limit sets of any half infinite orbit is either a singular point, a closed orbit or a \emph{polycycle} consisting of orbits between between singularities.
\begin{lem}[\cite{Gir}, Proposition 2.5]\label{Poin_Ben}
Let $S$ be a closed surface in a contact manifold and suppose that $\xi(S)$ has the Poincar\'e-Bendixson property. Then $S$ is convex if and only if all closed orbits are non-degenerate (i.e.\ the return map $\phi$ has $\phi'(0) \neq 1)$ and there are no oriented connections from negative to positive singularities.
\end{lem}
\begin{ex}\label{torus_ex}
Since all flows on planar surfaces satisfy the Poincar\'e-Bendixson property, it follows that any singular foliation with isolated singularities on $S = T^2$ with at least one closed orbit automatically has the Poincar\'e-Bendixson property. If in addition all singularities are positive, then the only way that $S$ cannot be convex is that it has a degenerate closed orbit.
\end{ex}




\section{Novikov Components and Reebless foliations}
Let $\mathcal{F}$ be a foliation on a closed $3$-manifold. For a leaf $L$, we define $A(L)$ as the set of points $y$ such that there exists a closed curve transverse to $\mathcal{F}$ intersecting $L$ that contains $y$. The set $A(L)$ is called the \emph{Novikov component} of $L$ and is an open (possibly empty) \emph{saturated} subset of $M$, whose boundary consists of closed leaves. Here saturated means that if $x \in A(L)$ then the entire leaf $L_x$ through $x$ also lies in $A(L)$. The boundary leaves of $A(L)$ are called \emph{barriers}, since they do not admit closed transversals. Let $L_0$ be such a barrier leaf. After possibly swapping the coorientation of $\mathcal{F}$ we may assume that $L_0$ is cooriented by the inward pointing normal of $A(L)$. We let $A_+(L_0)$ be the subset of all those points that are reachable from $L_0$ by an arc positively transverse to $\mathcal{F}$. Note that $A_+(L_0)$ is an open submanifold whose boundary again consists of closed leaves, which are all oriented by the inward pointing normal and $L_0 \subset \partial A_+(L_0)$.
By doubling $A_+(L_0)$, one obtains a closed oriented manifold $M_{Double}$. The normal directions to $\mathcal{F}$ endow $A_+(L_0)$ with a nowhere vanishing vector field that is inward pointing on the boundary. Thus we have the following formula for the Euler characteristic 
$$0 = \chi(M_{Double}) = 2\chi(A_+(L_0)) - \chi({\partial A_+(L_0)}) = - \chi({\partial A_+(L_0)}).$$
Since none of the boundary components of $A_+(L_0)$ can be spheres by the Reeb Stability Theorem, it follows immediately that all boundary components must have Euler characteristic zero. Thus all barrier leaves must be tori, a fact that goes back to Goodman \cite{Go}.

By a result of Haefliger the set of torus leaves is compact. It follows either that there are finitely many disjoint embeddings $N_i \cong T^2 \times [0,c_i] \to M$ so that $\partial N_i$ consists of leaves and $M \setminus \cup \, N_i$ has no torus leaves or $\mathcal{F}$ is itself a foliation by tori. We call each subset $N_i$ a \emph{stack} of torus leaves. Note that we allow $c_i = 0$ in which case $N_i$ consists of a single torus leaf. We may also assume that $\mathcal{F}$ is transverse to the intervals $\{pt\} \times [0,c_i]$ on each $N_i$. If $\mathcal{F}$ is not a foliation by tori we let $N_{Tor} =\cup \thinspace N_i$ and note that $\mathcal{F}$ is taut on $M_{taut} = M \setminus N_{Tor}$. In the case that $\mathcal{F}$ is a foliation by tori we set $M_{taut} = M$, since in this case $\mathcal{F}$ is itself taut.

We thicken the neighbourhoods $N_i$ to obtain $\widehat{N}_i \cong T^2 \times [-\epsilon,c_i+ \epsilon]$ so that the boundary tori are transverse to $\mathcal{F}$.
The way the foliation can look near the boundary of a stack $N_i$ is very restricted. In fact, results of Kopell and Szekeres imply that after a suitable isotopy one can assume that the induced foliation on tori near the ends of a stack of torus leaves is \emph{linear}. More precisely, we have the following (cf.\ \cite{Eyn}, Lemme 5.21):
\begin{lem}\label{szek_lem}
Let $\mathcal{F}$ be a smooth foliation on $T^2 \times [0,\epsilon]$ having only $L = T^2 \times \{0\}$ as a closed leaf and let $(x,y,z)$ denote standard coordinates on $T^2 \times [0,\epsilon]$. Then there is a fiber-preserving $C^1$-isotopy $\phi_t$ mapping $T^2 \times [0,\epsilon]$ into itself that fixes $L$ and is smooth away on $T^2 \times (0,\epsilon]$ such that the image of $\mathcal{F}$ under $\phi_1$ is defined by the kernel of the $1$-form
$$dz - u(z)(a \thinspace dx + b \thinspace dy)$$
for some function $u(z) \geq 0$ that is positive away from $z=0$ and $a,b \in \mathbb{R}$.
\end{lem}
\noindent We apply Lemma \ref{szek_lem} near each end of a stack. Then after being normalised to have length $1$ the pair $(-a,-b)$ corresponds to the (signed) slope near an end of such a stack. If the slopes at the two ends of a given stack of leaves are not equal, then this stack is \emph{stable}. This terminology stems from the fact that any foliation that is $C^0$-close to $\mathcal{F}$ has a closed torus leaf in a neighbourhood of the given stack. If the slopes agree, then the stack is called \emph{unstable}. 

Note that since the foliation $\mathcal{F}$ is transverse both to the intervals of $\widehat{N}_i \cong T^2 \times [-\epsilon,c_i+ \epsilon]$ and to its boundary, $T\mathcal{F}$ has trivial relative Euler class on $\widehat{N}_i$ and this is of course $C^0$-stable. 
We next claim that any transverse contact structure $\xi$ on $\widehat{N}_i$ is universally tight.
\begin{lem}\label{hor_univ_tight}
Let $\xi$ be a contact structure on $T^2 \times[0,1]$ that is transverse to the interval fibers and assume that $\xi$ is transverse to the boundary tori and that the characteristic foliations on each boundary torus is diffeomorphic to the suspension of an $S^1$-diffemorphism. Then $\xi$ is universally tight.
\end{lem}
\begin{proof}
After an initial $C^{\infty}$-small perturbation we may assume without loss of generality that the boundary tori are convex. We let $\frac{\partial}{\partial t}$ denote the coordinate vector field given by the second coordinate in the product $T^2 \times[0,1]$, which is in particular transverse to $\xi$. Near the ends we choose contact vector fields $X_0,X_1$ that are transverse to $\partial (T^2 \times[0,1])$. We let $s$ be the coordinate given by the flow of $X_0$ resp.\ $X_1$. Then near $T^2 \times \{0\}$ and $T^2 \times \{1\}$ respectively the contact structure is the kernel of a $1$-form
$$\lambda_i = \beta_i + f_ids, i =0,1 \, $$
where $ \beta_i$ and $f_i$ are independent of $s$. Thus we may add on half infinite ends to obtain a contact structure $\xi'$ on $M = T^2 \times \mathbb{R}$ that is $s$-invariant oustide $M = T^2 \times [0,1]$. Since the characteristic foliations near the boundary tori are given by suspension foliations, we may choose linear vector fields on $T^2$
$$a_i\frac{\partial}{\partial x} + b_i\frac{\partial}{\partial y}$$
with rational slopes $\frac{b_i}{a_i} \in \mathbb{Q}$ that are transverse to the characteristic foliations on $T^2 \times\{0\}$ and $T^2 \times\{1\}$ respectively. For sufficiently large $C$ the vector fields
$$Y_i =  \frac{\partial}{\partial s} + C(a_i\frac{\partial}{\partial x} + b_i\frac{\partial}{\partial y})$$
are then transverse to $\xi'$ outside some compact set. By using a partition of unity we may then extend $\frac{\partial}{\partial t}$ to a vector field $Y$ that agrees with $Y_0$ and $Y_1$ near the ends and is still transverse to $\xi'$.

Considering the identification $M \cong T^2 \times \mathbb{R}$ given by the flow of $Y$ makes $\xi'$ everywhere transverse to the $\mathbb{R}$-fibers and we denote the $\mathbb{R}$-coordinate by $t$. Note that the flow induced by $Y$ is just translation in the $s$-direction composed with a periodic linear flow 
$$\Phi^{t}_i(x,y) = (x,y) + t(Ca_i,Cb_i)$$
in the torus direction for $|t|$ sufficiently large. Using the translation invariance of $\xi'$ near the ends with respect to the $s$-coordinate and the periodicity of the flow $\Phi^{t}_i$, we deduce that the contact structure is periodic in the $t$ coordinate for $|t|$ sufficiently large. It follows that $\xi'$ defines a complete connection on $T^2 \times \mathbb{R}$ and that $\xi'$ is universally tight. To see this latter claim one takes the pull-back to the universal cover $\mathbb{R}^2 \times \mathbb{R}$ of $T^2 \times \mathbb{R}$ and lifts the family of curves $\{y = pt\}$ using the completeness of the connection. This gives coordinates $(x,y,z)$ on $\mathbb{R}^2 \times \mathbb{R}$ so that the contact structure is the kernel of
$$dz + f(x,y,z)dy.$$
By the contact condition $\frac{\partial f}{\partial x} > 0$, thus choosing new coordinates $(x',y',z') = (f(x,y,z),y,z)$ we see that $\xi'$ is contactomorphic to the standard contact structure on $\mathbb{R}^3$, which is tight. It follows in particular that $\xi$ itself is universally tight.
\end{proof}
\subsection{Movies of transverse contact structures}
Giroux \cite{Gir} has classified contact structures on the product $T^2 \times[0,1]$ by considering the \emph{movies} given by the family of characteristic foliations on the torus slices $F_t = T^2 \times \{t\}$. Although it is in general difficult to describe precisely which movies occur, Giroux proved the existence of a normal form for tight contact structures with certain boundary constraints, which then yields a classification up to isotopy relative to the boundary. One of the key points in Giroux's classification is that the isotopy class of a tight contact structure on $T^2 \times[0,1]$ is essentially determined by its ``feuilles'' or sheets.
\begin{defn}
Let $\xi$ be a contact structure on $T^2 \times[0,1]$. A \textbf{sheet} is a properly embedded surface $S$, such that the intersection $S \cap F_t$ is either empty or a smooth Legendrian curve and all singularities of $\xi(F_t)$ have the same sign. The collection of all sheets is called the \textbf{feuillage} associated to $\xi$.
\end{defn}
Since we are only interested in contact structures up to isotopy, we will always assume that all movies are $C^1$-generic. For us this will mean that all closed orbits and singularities of the characteristic foliations $\xi(F_t)$ are isolated and are non-degenerate or of birth-death type and that there is at most one degenerate orbit/singularity of birth-death type for each level $F_t$. Here a closed orbit is non-degenerate if the return map $\phi(t)$ has non-trivial linear holonomy. A non-degenerate closed orbit is \emph{repelling} if $\phi'(0)>1$ and it is \emph{attractive} if $\phi'(0)<1$, where the return map is taken in the direction determined by the induced orientation on $\xi(F_t)$ and $0$ corresponds to the closed orbit. Note that a non-degenerate closed orbit of $\xi(F_t)$ can be realised as the intersection of a small annulus $A$ with $F_t$ so that the intersection of $A$ with nearby levels is also a non-degenerate closed orbit. We will also always assume that the characteristic foliations on the boundary tori $F_0,F_1$ are given by suspension foliations.

The key classification results for contact structures on thickened tori are due to Giroux and independently Honda \cite{Hon}. The following is a special case of (\cite{Gir}, Th\'eor\`eme 1.5) that is tailored to our needs.
\begin{thm}\label{Gir_class}
Let $\xi$ be a universally tight contact structure on $T^2 \times [0,1]$ and assume that the characteristic foliations on the boundary are given by suspension foliations with precisely two non-degenerate closed orbits and that the relative Euler class of $\xi$ is trivial. Then $\xi$ is isotopic to a contact structure whose associated movie consists of suspension foliations.
 
\end{thm}
\noindent Contact structures whose movies consist of suspension foliations are referred to as ``rotatives" by Giroux in \cite{Gir}. A key tool in manipulating the movies of a given contact structure is the so-called elimination lemma of Giroux and Fuchs (see e.g.\ \cite{Gei}, Lemma 4.6.26).
\begin{lem}[Elimination Lemma]\label{elim_lem}
Let $e_+,h_+$ be non-degenerate positive elliptic resp.\ hyperbolic singularities of the characteristic foliation on $F_{t_0}$ and assume they are connected by some trajectory $E$ of $\xi(F_{t_0})$. Then there is a $C^0$-small isotopy with support in a neighbourhood of $E$ eliminating the pair of singularities.

Moreover, if $U$ is some compact neighbourhood of $E$ such that $d\alpha|_{F_{t_0}} > 0$ for some contact form $\alpha$, then this is also true after elimination.
\end{lem}
\begin{rem}
Note that the Elimination Lemma in \cite{Gir}, gives criteria to ensure that one can eliminate singularities without altering convexity. The version above is much weaker than this and the fact that $d\alpha|_{F_{t_0}} > 0$ can be assumed to be positive is an immediate consequence of the way the elimination is carried out.
\end{rem}
\noindent Another key ingredient in manipulating the movie of a contact structure is Giroux's Flexibility Lemma (\cite{Gir}, Lemme 2.7).
\begin{lem}[Flexibility Lemma]
Let $\xi,\xi'$ be contact structures on $N =T^2 \times [0,1]$ such that all levels $F_t$ are convex and such that the characteristic foliations agree on the boundary. Assume that both $\xi$ and $\xi'$ are divided by a continuous family of curves $\Gamma_t$. Then $\xi$ is isotopic to $\xi'$ relative to $\partial N$.
\end{lem}
\noindent There is also a relative version of the Flexibility Lemma, as stated in (\cite{Vog2}, Lemma 3.4).
\begin{lem}[Relative Flexibility Lemma]
Let $\xi$ be a contact structure on $N =T^2 \times [0,1]$ and let $F'_t$ be a smooth collection of compact subsurfaces such that $\xi(F'_t)$ is transverse to $\partial F'_t$ and $d\alpha|_{F'_t} >0$ for some contact form $\alpha$. Let $\lambda_t$ be a family of $1$-forms agreeing with $\alpha|_{F'_t}$ near $\partial F'_t$ and $\alpha|_{F'_t}$ for $t=0,1$ and such that $d\lambda_t|_{F'_t} > 0$. Then $\xi$ is isotopic to a contact structure whose characteristic foliation is given by $\lambda_t$ on $\cup_t \thinspace F'_t$ and which agrees with $\xi$ outside this set.
\end{lem}
Building on Giroux's work Vogel \cite{Vog2} has shown that the sheets of a contact structure that is transverse to the interval fibers of $T^2 \times [0,1]$ are of a very restricted form. The key observation is that any attractive closed orbit is contained in an annular sheet consisting entirely of coherently oriented closed orbits which is transverse to the interval fibers and either terminates at the boundary or at a degenerate closed orbit.

\begin{lem}[\cite{Vog2}, Proposition 3.14, Lemma 3.16]\label{sheets_attractive}
Let $\xi$ be a contact structure on $T^2 \times [0,1]$  that is positively transverse to the interval fibers and let  $\beta$ be an attractive orbit of the characteristic foliation $\xi(F_t)$. Then $\beta$ lies on an open annular sheet $A(\beta)$ consisting of attractive closed orbits which can be compactified to a closed annulus such that each boundary component is either contained in the boundary of $N$ or is a degenerate closed orbit of $\xi(F_{t'})$ for some $0<t' < 1$.

Moreover, the map given by projecting $A(\beta)$ to $T^2$ is a submersion and as $t$ increases $A(\beta) \cap F_t$ moves in the direction opposite to that determined by the coorientation of $\xi(F_t)$.
\end{lem}

\begin{figure}[t]
\psset{xunit=.4pt,yunit=.4pt,runit=.4pt}
\begin{pspicture}(928.8125,380.4125061)
{
\newrgbcolor{curcolor}{0 0 0}
\pscustom[linewidth=5.05960894,linecolor=curcolor]
{
\newpath
\moveto(2.54139709,377.89663696)
\lineto(406.28431702,377.89663696)
\lineto(406.28431702,-10.67480469)
\lineto(2.54139709,-10.67480469)
\closepath
}
}
{
\newrgbcolor{curcolor}{0 0 0}
\rput(-30,300){$t_{max}$}\rput(40,160){$A_{+}$} \rput(210,130){$A_{-}$} \rput(-30,-10){$F_0$}\rput(-30,380){$F_1$}\pscustom[linewidth=6,linecolor=curcolor]
{
\newpath
\moveto(51.55918,-8.5245939)
\curveto(62.18201,71.0091061)(64.835,89.3364061)(99.03635,169.2215861)
\curveto(109.86048,189.3162961)(132.52833,229.3023961)(160.72265,238.8683161)
\curveto(170.65006,242.2365461)(179.70324,244.8073661)(187.98609,244.5290461)
}
}
{
\newrgbcolor{curcolor}{0 0 0}
\pscustom[linewidth=1,linecolor=curcolor]
{
\newpath
\moveto(187.98609,244.5290461)
\curveto(213.55415,243.6699261)(231.75173,229.9678161)(245.50847,192.4550861)
\curveto(249.84987,157.3737461)(239.3286,122.5676161)(217.2242,87.3992061)
\curveto(212.11993,81.2302061)(207.79791,78.9869061)(203.97863,79.5314061)
}
}
{
\newrgbcolor{curcolor}{1 0 0}
\pscustom[linewidth=6,linecolor=curcolor]
{
\newpath
\moveto(203.97863,79.5314061)
\curveto(185.82687,82.1192061)(179.0309,147.6783861)(153.5846,154.0692961)
}
}
{
\newrgbcolor{curcolor}{0 0 0}
\pscustom[linewidth=1,linecolor=curcolor]
{
\newpath
\moveto(153.5846,154.0692961)
\curveto(123.68603,143.3819761)(123.73198,122.3944161)(126.31048,91.4398061)
\curveto(138.15658,38.9414061)(168.44466,23.3288061)(190.96644,23.3872061)
}
}
{
\newrgbcolor{curcolor}{0 0 0}
\pscustom[linewidth=6,linecolor=curcolor]
{
\newpath
\moveto(190.96644,23.3872061)
\curveto(205.18108,23.4242061)(214.29367,39.8050061)(222.29526,46.9931061)
\curveto(246.23238,86.6016061)(254.13596,120.6524061)(265.41928,147.5452161)
\curveto(273.46822,166.7291761)(278.91854,188.0421261)(286.44195,206.9973261)
\curveto(292.90074,223.2702061)(299.86122,247.3831961)(304.43316,256.4004061)
\curveto(310.70381,268.7679661)(314.83316,291.7077561)(329.39174,289.4297361)
}
}
{
\newrgbcolor{curcolor}{0 0 0}
\pscustom[linewidth=1,linecolor=curcolor]
{
\newpath
\moveto(329.39174,292.4398861)
\curveto(405.2403,295.2524161)(395.67272,-70.1896939)(396.0009,-4.5449939)
}
}
{
\newrgbcolor{curcolor}{0 0 0}
\pscustom[linewidth=1,linecolor=curcolor,linestyle=dashed,dash=3 3]
{
\newpath
\moveto(1.051555,245.9931661)
\lineto(405.11257,244.9830161)
}
}
{
\newrgbcolor{curcolor}{0 0 0}
\pscustom[linewidth=1,linecolor=curcolor,linestyle=dashed,dash=3 3]
{
\newpath
\moveto(1.051555,291.9931661)
\lineto(405.11257,290.9830161)
}
}
{
\newrgbcolor{curcolor}{0 0 0}
\pscustom[linewidth=5.05960894,linecolor=curcolor]
{
\newpath
\moveto(522.54138184,377.89663696)
\lineto(926.28430176,377.89663696)
\lineto(926.28430176,-10.67480469)
\lineto(522.54138184,-10.67480469)
\closepath
}
}
{
\newrgbcolor{curcolor}{0 0 0}
\pscustom[linewidth=6,linecolor=curcolor]
{
\newpath
\moveto(571.55918,-8.5245939)
\curveto(582.18201,71.0091061)(584.835,89.3364061)(619.03635,169.2215861)
\curveto(629.86048,189.3162961)(652.52833,229.3023961)(680.72265,238.8683161)
\curveto(690.65006,242.2365461)(699.70324,244.8073661)(707.98609,244.5290461)
}
}
{
\newrgbcolor{curcolor}{0 0 0}
\pscustom[linewidth=1,linecolor=curcolor]
{
\newpath
\moveto(707.98609,244.5290461)
\curveto(733.55415,243.6699261)(751.75173,229.9678161)(765.50847,192.4550861)
\curveto(769.84987,157.3737461)(759.3286,122.5676161)(737.2242,87.3992061)
\curveto(732.11993,81.2302061)(727.79791,78.9869061)(723.97863,79.5314061)
}
}
{
\newrgbcolor{curcolor}{0 0 0}
\pscustom[linewidth=6,linecolor=curcolor]
{
\newpath
\moveto(710.96644,23.3872061)
\curveto(725.18108,23.4242061)(734.29367,39.8050061)(742.29526,46.9931061)
\curveto(766.23238,86.6016061)(774.13596,120.6524061)(785.41928,147.5452161)
\curveto(793.46822,166.7291761)(798.91854,188.0421261)(806.44195,206.9973261)
\curveto(812.90074,223.2702061)(819.86122,247.3831961)(824.43316,256.4004061)
\curveto(830.70381,268.7679661)(834.83316,291.7077561)(849.39174,289.4297361)
}
}
{
\newrgbcolor{curcolor}{0 0 0}
\pscustom[linewidth=1,linecolor=curcolor]
{
\newpath
\moveto(849.39174,290.4398861)
\curveto(925.24035,293.2524161)(915.67275,-70.1896939)(916.00085,-4.5449939)
}
}
{
\newrgbcolor{curcolor}{0 0 0}
\pscustom[linewidth=1,linecolor=curcolor,linestyle=dashed,dash=3 3]
{
\newpath
\moveto(521.05155,245.9931661)
\lineto(925.11255,244.9830161)
}
}
{
\newrgbcolor{curcolor}{0 0 0}
\pscustom[linewidth=1,linecolor=curcolor,linestyle=dashed,dash=3 3]
{
\newpath
\moveto(521.05155,291.9931661)
\lineto(925.11255,290.9830161)
}
}
{
\newrgbcolor{curcolor}{0 0 0}
\pscustom[linewidth=1,linecolor=curcolor]
{
\newpath
\moveto(710.57964,23.7596061)
\curveto(689.51009,22.8766061)(665.08944,25.1631061)(660.07201,51.0337061)
\curveto(665.59163,66.4567061)(698.94462,75.5289061)(725.71163,80.3282061)
}
}
\end{pspicture} \caption{A sheet before and after (a partial) straightening. The thickened segments correspond to attractive closed orbits oriented positively resp.\ negatively ($A_+$/$A_-$) and degenerate orbits correspond to tangencies of sheets with the levels of $T^2\times [0,1]$.}\label{sheet_fig}
\end{figure}
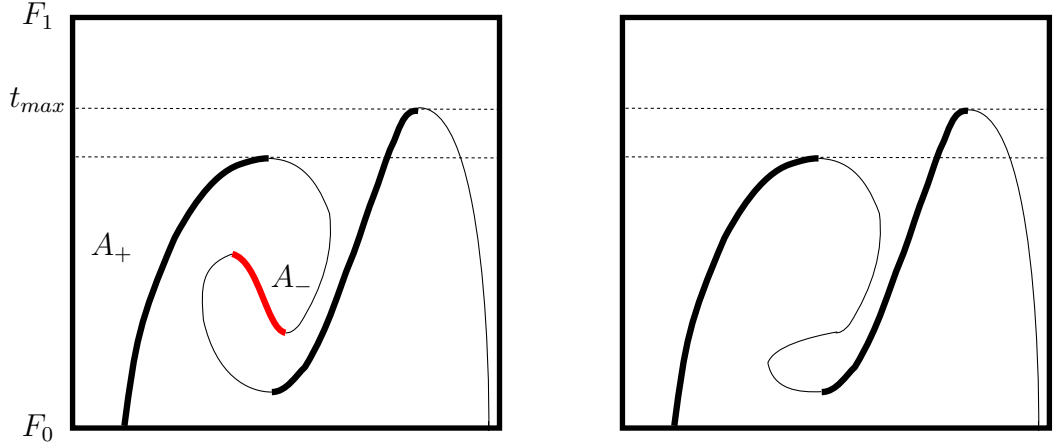

\noindent Although the way a sheet lies in the product  $N =T^2 \times [0,1]$ can be quite complicated they can always be straightened so that they have at most one tangency with the levels of $N$ (cf.\ \cite{Gir}, Proposition 3.22). This straightening is particularly easy to describe in the case that the contact structure is \emph{elementary}. This means that there are coordinates $(x,y,t) \in S^1 \times S^1 \times [0,1] = N$ so that $S^1 \times (y ,t)$ is either Legendrian (i.e.\ tangent) or transverse to the contact structure $\xi$. In this case straightening sheets amounts to straightening a collection of properly embedded arcs given by projecting to $S^1 \times [0,1]$.
\begin{lem}[Straightening sheets]\label{straight_lem}
Let $\xi$ be an elementary contact structure on $N =T^2 \times [0,1]$. Let $A$ be an annular sheet such that $\partial A$ lies in a single boundary component of $\partial N$. Then after an isotopy we may assume that $A$ has precisely one tangency with the levels of $N$ (corresponding to a degenerate closed orbit of the characteristic foliation). If the boundary components of $A$ lie in different components of $\partial N$, then we may isotope $A$ to be transverse to all levels of $N$.

Moreover, this isotopy can be made relative to all other sheets.

\end{lem}
\noindent Using  Lemma \ref{sheets_attractive} one finds a movie of a particularly nice form for any transverse contact structure. 
\begin{prop}\label{Giroux}
Let $\xi$ be a contact structure on $T^2 \times[0,1]$ that is (positively) transverse to the interval fibers.  Assume that the boundary tori are convex with characteristic foliations given by suspensions of $S^1$-diffeomorphisms. Then $\xi$ is isotopic relative to the boundary of $T^2 \times[0,1]$ to a contact structure that is at all times transverse to the tori $F_t = T^2 \times \{t\}$ and the characteristic foliations $\xi(F_t)$ are given by suspensions.
\end{prop}
\begin{proof}
First suppose that on all levels the characteristic foliation $\xi(F_t)$ has at least one (attractive) closed orbit. In particular, this means that $\xi(F_t)$ has the Poincar\'e-Bendixson property for $0 \leq t \leq 1$ (cf.\ Example \ref{torus_ex}). Since the contact structure is (positively) transverse to the interval fibers, all singularities are positive so that the only way that a surface $F_t = T^2 \times\{t\}$ cannot be convex is that it contains one of the finitely many degenerate closed orbits by Lemma \ref{Poin_Ben}. We let $0 < t_0 < t_1<\ldots < t_k < 1$ be the non-convex levels. One can then use the Elimination Lemma to eliminate all singularities on $F_{t_i}$ via an isotopy with support near $F_{t_i}$ (by genericity all singularities are non-degenerate on $F_{t_i}$). Moreover, since all singularities are positive and $F_{t_i}$ has isolated closed orbits, we may assume that the trajectories connecting them are contained in a collection of annular neighbourhoods $U_j$ such that $\partial U_j$ is transverse to the characteristic foliation and $d\alpha|_{U_j}>0$ for a contact form $\alpha$. Thus these eliminations can be chosen to have support disjoint from the closed orbits of $F_t$ for $t$ near $t_i$ and can be performed without introducing either degenerate resp.\ attractive closed orbits or negative singularities. Thus all levels $F_t$ near $F_{t_i}$ with $t\neq t_i$ remain convex.

We can now assume that we have eliminated all singularities on each non-convex level. Then using the Flexibility Lemma as in (\cite{Gir}, Proposition 3.15) we can isotope $\xi$ relative to $F_0 \sqcup F_1$ to become elementary. In particular, we obtain a movie so that each closed orbit of $F_{t_i + \epsilon}$ is connected to a closed orbit of $F_{t_{i+1} - \epsilon}$ by an annular sheet $A$ that is transverse to each level $F_t$. The union of these annuli together with the closed orbits near each non-convex level $F_{t_i}$ then give the feuillage of $\xi$. Note that this can be achieved without altering the contact structure near the non-convex levels. In view of Lemma \ref{sheets_attractive} we can also assume that we did not alter the contact structure near any attractive closed orbits. Furthermore, these modifications are such that on the parts of a sheet not consisting of attractive closed orbits the Legendrian curves $\gamma_t = A \cap F_t$ are all repulsive closed orbits, except in the case that the orientations of the closed orbits at each end of the annulus are opposite, in which case $\gamma_{t'}$ is a singular circle of $\xi(F_{t'})$ for precisely one value $t_i<t'<t_{i+i}$. 

Now all sheets are properly embedded and we consider a sheet $A$ whose negative end intersects $F_0$ in an attractive closed orbit. First assume that both boundary components of $A$ are contained in $F_0$. Let $t_{max}$ be the maximum value of $t$ such that $A \cap F_{t_{max}}$ is non-empty. Then since the direction that a sheet moves near an attractive closed orbit is determined by its (co)orientation, we see that the orientation of $A \cap F_{t_{max}}$ must agree with that of $A \cap F_{0}$ (cf.\ Figure \ref{sheet_fig}). A symmetric argument applies to sheets whose ends lie in $F_1$ (in which case one considers $t_{min}$).  In a similar way a sheet whose boundary intersects both boundary components and begins at an attractive closed orbit must intersect the other boundary component in an attractive closed orbit with the same orientation. We then straighten out sheets to assume that each sheet is transverse to all levels $F_t$ or has precisely one point of tangency, which is then a degenerate closed orbit whose orientation agrees with that of the closed orbit on the boundary of the sheet.

We first assume that there is a sheet that intersects both boundary components. We may then apply the Relative Version of the Flexibility Lemma away from the degenerate closed orbits to assume that all closed orbits in the movie are in fact oriented in the same direction. In particular, the non-convex levels are given by suspension foliations. A final application of the Flexibility Lemma provides an isotopy to a contact structure whose associated movie consists entirely of suspension foliations.

In the case that no sheet intersects both boundary components, we can argue as above until the first point $t_- \in [0,1]$ where all sheets beginning at $T^2 \times \{0\}$ have disappeared. We also let $t_+$ be the value at which all sheets beginning at $T^2 \times \{1\}$ have disappeared. By isotoping the sheets beginning at $F_0$ downwards, we may assume that $t_- < t_+$. After a further isotopy, we may also assume that for some small $\epsilon > 0$ the tori $F_{t_- - \epsilon},F_{t_+ + \epsilon}$ are convex and have only $2$ closed orbits. Since the contact structure is transverse to the interval fibers, its relative Euler class on $T^2 \times [0,1]$ is trivial. Furthermore, since the contact structure is transverse to the $T^2$-slices on $T^2 \times [0,t_- - \epsilon]$ and $T^2 \times [t_+ + \epsilon,1]$, the relative Euler class is trivial on each of these pieces and hence it is also trivial on $T^2 \times [t_- - \epsilon,t_+ + \epsilon]$. The result then follows by Theorem \ref{Gir_class}, since $\xi$ is universally tight in view of Lemma \ref{hor_univ_tight}.
\end{proof}
\begin{rem}\label{Giroux_rem}
The contact structures given in Proposition \ref{Giroux} are of two kinds, either the asymptotic slopes of the characteristic foliations $\xi(F_t)$ are constant or not. In the latter case there is a time $t_-$ when all sheets intersecting $F_0$ first disappear. Near to $F_{t_-}$ one can find a linearly foliated torus $F_{t_- + \epsilon}$. The same is true for the time $t_+$ at which all sheets that intersect $F_1$ have disappeared. Moreover, after an isotopy one can assume that the characteristic foliations are all transverse to non-vanishing closed $1$-forms $\alpha_0,\alpha_1$ on $T^2 \times [0,t_- + \epsilon]$ and $T^2 \times [t_+-\epsilon,1]$ respectively. If the asymptotic slope is constant, then we can assume that the characteristic foliations $\xi(F_t)$ are all transverse to a \emph{fixed} non-vanishing closed $1$-form for all $t \in[0,1]$, provided this holds on the boundary.
\end{rem}
\begin{rem}
Using Remark \ref{Giroux_rem} it is easy to construct a closed dominating $2$-form for any $C^0$-small perturbation of a Reebless foliation $\mathcal{F}$ which is either transverse to a taut foliation or has the form $\alpha_t \wedge dt$ on each thickened stack $\widehat{N}_i \cong T^2 \times [0,1]$ where $\alpha_t$ are non-vanishing closed $1$-forms that depend smoothly on $t$. Since all torus leaves of a Reebless foliation are incompressible and no transversal of a taut foliation can be contractible, it follows that $\xi$ is taut in the sense of (\cite{ETh}, Definition 3.5.3). So the (universal) tightness of perturbations of Reebless foliations would be implied by the fact that tautness implies (universal) tightness for contact structures (cf.\ \cite{ETh}, Conjecture 3.5.14).
\end{rem}
We shall need two more ingredients for the proof of Theorem \ref{Reebless}. The first is Colin's result on glueing contact structures along linearly foliated (pre-Lagrangian) tori.
\begin{thm}[\cite{Col0}, Th\'eor\`eme 4.2]\label{Colin}
Let $\xi$ be contact structure on a manifold $M$ possibly with boundary and let $T$ be an incompressible torus in the interior of $M$ such that $\xi(T)$ is linear. If the restriction of $\xi$ to $M \setminus T$ is universally tight, then $\xi$ is also universally tight on all of $M$.
\end{thm}
We will also need a version of Theorem \ref{taut _fill} for taut foliations on manifolds with boundary. We will state this in a slightly more technical fashion, which is best suited for the proof of Theorem \ref{Reebless} below. The proof is essentially the same as in the closed case except that one needs now to take care near the boundary by completing the manifold in a controlled fashion
\begin{thm}\label{taut _fill_bound}
Let $\mathcal{F}$ be a taut foliation that is transverse to $\partial M$ (if non-empty) and assume that the induced $1$-dimensional foliation on the boundary is also taut. Let $\omega$ be a dominating $2$-form which has the form $\alpha \wedge dt$ near $\partial M$, where $\alpha$ is a closed $1$-form on $\partial M$. Then any contact structure $\xi$ that is dominated by $\omega$ is universally tight. 
\end{thm}
\begin{proof}
After a suitable choice of coordinates we assume that $\mathcal{F}$ is a product foliation near the boundary. We may then take the double of $M$ which carries a (taut) foliation given by doubling $\mathcal{F}$. Thus according to Theorem \ref{Eli_Th} there is a negative contact structure $\xi_-$ on $M$ that is $C^0$-close to $\mathcal{F}$. Now after a $C^{\infty}$-small perturbation we may assume that all boundary tori are convex with respect to both $\xi$ and $\xi_-$ respectively. We next attach half infinite ends $T^2 \times [0,\infty)$ to each boundary component to obtain a completion $\widehat{M}$ of $M$.

Any smooth extension of $\xi$, again denoted $\xi$, will be contact in a small neighbourhood of $M$ in $\widehat{M}$. Then on a suitable neighbourhood $N \cong T^2 \times (-\epsilon,\epsilon)$ of a boundary component $T_i \subset \partial M$ the contact structure $\xi$ is defined by a form
$$\lambda  = \beta_t + f_t dt$$
for a smooth family of $1$-forms $\alpha_t$ on $T^2$. The contact condition is then
$$f_t d\beta_t + \beta_t \wedge (d f_t - \dot{\beta}_t )> 0.$$
Since $T_i$ is convex, there is a function $g$ on $T^2$ such that
$$g d\beta_0+ \beta_0 \wedge d g > 0.$$
By taking $\epsilon$ small enough the same then holds on the neighbourhood $T^2 \times (-\epsilon,\epsilon)$. Let $\varphi(t)$ be a non-increasing cut off function that is identically $1$ for $t \leq -\frac\epsilon2$ and identically $0$ for $t \geq 0$. We then set
$$\widehat{\lambda} = \beta_{t \varphi(t-\frac\epsilon2)} +  \left(\varphi\thinspace f_t + K \thinspace (1-\varphi)g \right)dt.$$
Since the contact condition is convex in $f_t$ the form $\widehat{\lambda}$ is contact for $K$ large. Note that $\widehat{\lambda}$ is $t$ invariant for $t \geq \frac\epsilon2$ thus we can extend $\xi$ to a contact structure $\widehat{\xi}$ that is translation invariant on the half-infinite ends. Furthermore, the characteristic foliations on these ends are $C^0$-close to $\xi(T_i)$ and thus are dominated by the closed $2$-form $\widehat{\omega}$ that agrees with $\omega$ on $M$ and is equal to $\alpha \wedge dt$ on $\partial M \times [0,\infty)$. The same holds for the analogously defined extension $\widehat{\xi}_-$.

We then define a symplectic form on $W =\widehat{M} \times [-1,1]$ by setting
$$\widehat{\Omega} = \widehat{\omega} + \epsilon \thinspace d(s \thinspace \widehat{\lambda})$$
where $\widehat{\lambda}$ is a defining form for $\widehat{\xi}$ and taking $\epsilon$ small. This form is translation invariant on the half-infinite torus pieces outside a compact set. We then choose a compatible almost complex structure $J$ that is also translation invariant outside a compact set and leaves both $\widehat{\xi}$ and $\widehat{\xi}_-$ invariant. The symplectic manifold $(W,\widehat{\Omega})$ is a weak symplectic (semi-)filling of $\widehat{\xi}$ and the metric $\widehat{\Omega}(\cdot,J \cdot)$ has bounded geometry. The same is true when we pass to the universal cover. Then exactly as in the case when $M$ is closed, the Gromov-Eliashberg argument using Bishop families of $J$-holomorphic discs then shows that $\widehat{\xi}$ is universally tight on $\widehat{M}$ and \emph{a fortiori} so is $\xi$. Note that bubbling cannot occur \emph{a priori}, since by Novikov's Theorem $\pi_2(M)=0$ or $M = S^2 \times S^1$ and $\mathcal{F}$ is the product foliation which cannot be approximated by contact structures at all.
\end{proof}
\subsection{Proof of Theorem \ref{Reebless}}
\begin{proof}
Let $\mathcal{F}$ be a Reebless foliation and let $M = M_{taut} \cup N_{Tor}$ be the decomposition of $M$ into the piece on which $\mathcal{F}$ is taut and its complement consisting of stacks of torus leaves. Consider a stack $N_i \cong T^2 \times [0,c_i]$ and as above let $ \widehat{N}_i \cong T^2 \times [-\epsilon,c_i+\epsilon]$ be a small neighbourhood of $N_i$ so that $\mathcal{F}$ is transverse to $\partial \widehat{N}_i$. After an initial isotopy we may assume that the induced foliations on the boundary of each $\widehat{N}_i$ is linear. Let $\omega$ be a dominating closed $2$-form on $M_{taut}$ and let $\alpha^{\pm}_i$ be closed $1$-forms that are positive on the induced foliations on the positive resp.\ negative end of $\widehat{N}_i$. Without loss of generality we may assume that $\omega = \alpha^{\pm}_i \wedge dt$ near each boundary component of $\widehat{N}_i$. We let $\mathcal{U}_0$ be a $C^0$-neighbourhood of $T\mathcal{F}$ such that any $\xi \in \mathcal{U}_0$ is still dominated by $\omega$ and such that $\xi$ is transverse to the interval fibers on each neighbourhood $\widehat{N}_i$. We also choose $\mathcal{U}_0$ small enough so that the slopes of the characteristic foliations near the boundary of any stable stack of leaves remain distinct. 

We let $\xi$ be any contact structure in $\mathcal{U}_0$ and we first consider those (necessarily unstable) stacks $N_{i_k}$ on which the contact structure $\xi$ is isotopic to one whose slopes remain constant. Then according to Remark \ref{Giroux_rem} we can assume that the contact structure is dominated by $\alpha^{+}_i \wedge dt = \alpha^{-}_i \wedge dt = \alpha_i \wedge dt $ on $\widehat{N}_{i_k}$. By filling in a product foliation on $\widehat{N}_{i_k} $ we obtain a foliation $\mathcal{F}'$ on $$M'_{taut} = M_{taut} \cup\left( \bigcup_{k=1}^p N_{i_k}\right)$$ which is transverse to $\partial M'_{taut}$ and has a dominating closed $2$-form $\omega'$. By construction $\mathcal{F}'$ is taut. Furthermore, $\xi$ is isotopic to a contact structure that is dominated by $\omega'$ on $M'_{taut}$ and thus $\xi|_{M'_{taut}}$ is universally tight by Theorem \ref{taut _fill_bound}.

Since $\xi$ is transverse to the interval fibers on each thickened stack $\widehat{N}_{i}$, the restriction $\xi|_{\widehat{N}_{i}}$ is universally tight by Lemma \ref{hor_univ_tight}. According to Remark \ref{Giroux_rem} we may assume that the characteristic foliations on $\partial \widehat{N}_{i}$ are linear for the remaining stacks.  Since $\mathcal{F}$ is Reebless, all torus leaves are incompressible. Thus $\xi$ is obtained by gluing universally tight contact structures along incompressible linearly foliated tori and thus it is universally tight on all of $M$ by Theorem \ref{Colin}.
\end{proof}
\begin{rem}
Theorem \ref{Reebless} implies that any foliation that can be $C^0$-approximated by overtwisted contact structures must contain a Reeb component. It would be interesting to find a vanishing cycle, and hence a Reeb component, as some sort of limit of a sequence of overtwisted discs.
\end{rem}
\section{Notions of tightness for confoliations}\label{notions_of_tightness}
Eliashberg and Thurston \cite{ETh} introduced the notion of a tight confoliation as a generalisation of both the notions of tightness for contact structures and Reeblessness for foliations.
\begin{defn}[Tightness for confoliations \cite{ETh}]
A confoliation $\xi$ on a $3$-manifold $M$ is called \textbf{tight} if for every embedded disc $D \subset M$ such that $\partial D$ is tangent to $\xi$ and $D$ is transverse to $\xi$ along $\partial D$ there exists an embedded disc $D'$ so that
\begin{itemize}
\item $\partial D' = \partial D$
\item $D'$ is tangent to $\xi$ and $e(\xi)[D\cup D'] = 0$.
\end{itemize}
\end{defn}
\noindent A confoliation that is not tight can always be approximated by overtwisted contact structures (cf.\ \cite{Vog}, p.\ 119).
Eliashberg and Thurston conjectured that tightness should imply that the Thurston-Bennequin inequalities hold. However, Vogel \cite{Vog} showed that this is not the case. He in particular constructed tight confoliations on $T^3$ that do not satisfy the Thurston-Bennequin inequalities so that all perturbations are necessarily overtwisted. This led him to define a more restrictive notion of tightness that he calls \emph{s-tightness}.
\begin{defn}[$s$-tight confoliations \cite{Vog}]
A confoliation $\xi$ on a $3$-manifold $M$ is called \textbf{s-tight} if the characteristic foliation on any generically embedded closed surface contains no \textbf{overtwisted stars}.
\end{defn}
We refer to \cite{Vog} for a precise definition of an overtwisted star, but the most important aspect of the definition for us is that as in the case of tightness, confoliations that are not $s$-tight can be approximated by overtwisted contact structures (\cite{Vog}, Theorem 6.9). Both tightness and $s$-tightness have universal analogues. Namely a confoliation is \emph{universally tight} resp.\ \emph{$s$-tight} if all its finite covers are tight resp.\ $s$-tight. Note that this definition agrees with the ordinary definition of universal tightness for contact structures since in view of Geometrisation all $3$-manifold groups are residually finite. One could of course define universal tightness for confoliations as tightness of the universal cover. However for $s$-tightness, where the definition involves a compact surface, it is not clear that the pullback of a non $s$-tight confoliation to the universal cover is $s$-tight or not.

Since Reeblessness for foliations ensures universal tightness for all sufficiently small contact perturbations we make the following definition.
\begin{defn}
A confoliation $\xi$ is called \textbf{perturbation tight}, or simply \textbf{p-tight}, if it has a $C^0$-neighbourhood $\mathcal{U}_0$ such that all (positive) contact structures $\xi' \in \mathcal{U}_0$ are tight.
\end{defn}
\noindent In view of Colin's $C^0$-stability result \cite{Col1}, $p$-tightness agrees with ordinary tightness for contact structures. Thus $p$-tightness generalises both Reeblessness for foliations and tightness for contact structure. Furthermore, in the case of foliations $p$-tightness is in fact equivalent to Reeblessness, since any foliation with a Reeb component has an overtwisted contact perturbation (cf.\ \cite{ETh}). Thus we have the following sequence of inclusions for confoliations
$$\{p\text{-tight}\} \subseteq \{s\text{-tight}\} \subsetneq \{\text{tight}\}.$$
Note that all notions of tightness above agree with ordinary tightness for contact structures and Reeblessness for foliations and the last inclusion is strict. However for the universal versions of the above properties the nature of each of these inclusions is unknown.
\begin{qu}
Do the notions of universal $p$-tightness, $s$-tightness, tightness agree for confoliations?
\end{qu}
There are several qualitative aspects of these definitions that are quite different. In particular, $p$-tightness is by definition $C^0$-stable, whereas tightness in the sense of Eliashberg and Thurston is not in view of Vogel's examples. In the case of $s$-tightness $C^0$-stability seems to be an important open questions, even for transitive confoliations. On the other hand no notion of tightness for confoliations is closed as one sees by turbulising a taut foliation along a transversal, which gives a $C^{\infty}$-deformation of foliations that ends at a foliation with Reeb components, which is not tight in any sense.

\end{document}